\documentclass{amsart}
\usepackage{amssymb}
\usepackage{graphicx}

\newtheorem{thm}{Theorem}
\newtheorem{cor}{Corollary}
\newtheorem{lem}{Lemma}

\newtheorem{rem}{Remark}

\newcommand{\es}{{\mathcal S}}

\newcommand{\D}{{\mathbb D}}

\def\be{\begin{equation}}
\def\ee{\end{equation}}

\newcommand{\bee}{\begin{enumerate}}
\newcommand{\eee}{\end{enumerate}}

\newcommand{\blem}{\begin{lem}}
\newcommand{\elem}{\end{lem}}
\newcommand{\bthm}{\begin{thm}}
\newcommand{\ethm}{\end{thm}}
\newcommand{\bcor}{\begin{cor}}
\newcommand{\ecor}{\end{cor}}
\newcommand{\beg}{\begin{example}}
\newcommand{\eeg}{\end{example}}
\newcommand{\begs}{\begin{examples}}
\newcommand{\eegs}{\end{examples}}
\newcommand{\bdefe}{\begin{defin}}
\newcommand{\edefe}{\end{defin}}
\newcommand{\bprob}{\begin{prob}}
\newcommand{\eprob}{\end{prob}}
\newcommand{\bei}{\begin{itemize}}
\newcommand{\eei}{\end{itemize}}

\newcommand{\bcon}{\begin{conj}}
\newcommand{\econ}{\end{conj}}
\newcommand{\bcons}{\begin{conjs}}
\newcommand{\econs}{\end{conjs}}
\newcommand{\bprop}{\begin{propo}}
\newcommand{\eprop}{\end{propo}}
\newcommand{\br}{\begin{rem}}
\newcommand{\er}{\end{rem}}
\newcommand{\brs}{\begin{rems}}
\newcommand{\ers}{\end{rems}}
\newcommand{\bo}{\begin{obser}}
\newcommand{\eo}{\end{obser}}
\newcommand{\bos}{\begin{obsers}}
\newcommand{\eos}{\end{obsers}}
\newcommand{\bpf}{\begin{pf}}
\newcommand{\epf}{\end{pf}}
\newcommand{\ba}{\begin{array}}
\newcommand{\ea}{\end{array}}
\newcommand{\beq}{\begin{eqnarray}}
\newcommand{\beqq}{\begin{eqnarray*}}
\newcommand{\eeq}{\end{eqnarray}}
\newcommand{\eeqq}{\end{eqnarray*}}

\begin{document}
\bibliographystyle{amsplain}

\title[Improvements of certain results of the class $\mathcal{S}$]{Improvements of certain results of the class $\boldsymbol{\mathcal{S}}$ of univalent functions}

\author[M. Obradovi\'{c}]{Milutin Obradovi\'{c}}
\address{Department of Mathematics,
Faculty of Civil Engineering, University of Belgrade,
Bulevar Kralja Aleksandra 73, 11000, Belgrade, Serbia}
\email{obrad@grf.bg.ac.rs}


\author[N. Tuneski]{Nikola Tuneski}
\address{Department of Mathematics and Informatics, Faculty of Mechanical Engineering, Ss. Cyril and Methodius
University in Skopje, Karpo\v{s} II b.b., 1000 Skopje, Republic of North Macedonia.}
\email{nikola.tuneski@mf.edu.mk}


\dedicatory{This paper is dedicated to the memory of our esteemed colleague, Derek K. Thomas, whose untimely passing occurred during the course of this work. We honor his legacy and the lasting impact of his achievements in the field of univalent functions.}

\begin{abstract}
For $f\in \mathcal{S}$, the class  univalent functions  in the unit disk $\mathbb{D}$ and given by $f(z)=z+\sum_{n=2}^{\infty} a_n z^n$ for $z\in \mathbb{D}$, we improve previous bounds for the second and third Hankel determinants in case when either $a_2=0,$ or $a_3=0$. We also improve an upper bound for the coefficient difference $|a_4|-|a_3|$  when $f\in \mathcal{S}$.
\end{abstract}

\maketitle

\medskip

\section{Introduction and preliminaries }

\medskip

As usual, let $\mathcal{A}$ be the class of functions $f$ which are analytic  in the open unit disc $\D=\{z:|z|<1\}$ of the form
\be\label{eq 1}
f(z)=z+a_2z^2+a_3z^3+\cdots,
\ee
and let $\mathcal{S}$ be the subclass of $\mathcal{A}$ consisting of functions that are univalent in $\D$.

\medskip
In recent years a great deal of attention has been given to finding upper bounds for the modulus of the second and third Hankel determinants $H_2(2)$ and $H_3(1)$, defined as follows who's elements are the coefficients of $f\in\mathcal{S}$ (see e.g. \cite{DTV-book}).
 \medskip

For $f\in\mathcal{S}$
\be\label{eq 2}
  H_2(2) = a_2a_4-a_3^2
\end{equation}
and
\begin{equation}\label{eq 3}
H_3(1) = a_3(a_2a_4-a_3^2) - a_4(a_4-a_2a_3)+a_5(a_3-a_2^2).
\end{equation}
\medskip

Almost all results have concentrated on finding bounds for $|H_2(2)|$ and $|H_3(1)|$ for subclasses of $\mathcal{S}$, and only recently has a significant bound been found for the whole class $\mathcal{S}$  and for both, $|H_2(2)|$ and $|H_3(1)|$ (see \cite{OT-2021} and \cite{OT-2024}). However finding exact sharp bounds remains an open problem.

\medskip

In their paper \cite{OTT} the authors gave the next results concerning coefficients bound and Hankel determinants
$|H_2(2)|$ and $|H_3(1)|$.

\medskip

\begin{thm}\label{th-1}
Let $f\in\es$ and be given by \eqref{eq 1} with $a_2=0$. Then
\begin{itemize}
\item[($i$)] $|a_3|\le1$,
\medskip
\item[($ii$)] $|a_4|\le\dfrac23=0.666\ldots$,
\medskip
\item[($iii$)]  $|a_5|\le \frac{5}{4}+\frac{1}{\sqrt{15}} = 1.508\ldots$,
\medskip
\item[($iv$)] $|H_2(2)|\le1$,
\medskip
\item[($v$)] $|H_3(1)|\le\dfrac{41}{20} = 2.05$.
\end{itemize}
\end{thm}

\medskip

We would like to point out that there is a mistake in the estimate of $|a_5|$ given in Theorem 2.1(iii) from \cite{OTT}. The estimate given in the case (iii) in the theorem above is the correct one.

\medskip

Similar results are given in \cite{OTT} for the case when $a_3=0$.
\medskip

\begin{thm}\label{th-2}
Let $f\in\es$ and be  given by \eqref{eq 1}, with $a_3=0$. Then
\medskip
\begin{itemize}
\item[($i$)] $|a_2|\le1$,
\medskip
\item[($ii$)] $|a_4|\le \dfrac{\sqrt{37}+13}{12}=1.59023\ldots$,
\medskip
\item[($iii$)]  $|a_5|\le \dfrac14\sqrt{\dfrac{757}{15}}+1 = 2.77599\ldots$,
\medskip
\item[($iv$)] $|H_2(2)|\le1.75088\ldots$,
\medskip
\item[($v$)] $|H_3(1)|\le 1.114596\ldots$.
\end{itemize}
\end{thm}

\medskip

A long standing problem in the theory of univalent functions is to find  sharp upper and lower bounds for
$|a_{n+1}|-|a_n|$, when $f\in\mathcal{S}$. Since  the Keobe function has coefficients $a_n=n$, it is
natural to conjecture  that $||a_{n+1}|-|a_n||\le1$. As early as 1933, this was shown to be false even when
$n=2,$ when Fekete and Szeg\"o \cite{FekSze33} obtained the sharp bounds
 $$ -1 \leq |a_3| - |a_2| \leq \frac{3}{4} + e^{-\lambda_0}(2e^{-\lambda_0}-1) = 1.029\ldots, $$
where $\lambda_0$ is the unique value of $\lambda$ in $0 < \lambda <1$, satisfying the equation $4\lambda =e^{\lambda}$. \\
Hayman \cite{Hay63} showed that if $f \in {\mathcal S}$, then $| |a_{n+1}| - |a_n| | \leq C$, where $C$ is
an absolute constant. The exact value of $C$ is unknown,  the best estimate to date
being $C=3.61\ldots$ (see Grinspan \cite{Gri76}), which because of the sharp estimate above when $n=2$, cannot be reduced
to $1$.

\medskip

In \cite{OTT} the authors treated the difference coefficients  $|a_4|-|a_3|$ for $f\in \mathcal{S}$, which improves the previous cited result of Grinspan when $n=3$ as follows

\begin{thm}\label{th-3}
Let $f\in\es$ and be given by \eqref{eq 1}. Then
\[|a_4|-|a_3| \le 2.1033299\ldots.\]
\end{thm}

For the proofs of the previously cited results, the authors mainly used the property of Grunsky coefficients
given in the book of N. A.Lebedev (\cite{Lebedev}). In the proofs of the results in this paper we also use the same tools but with different approach that brings improved estimates.

\medskip

We proceed with the notations and results to be used.

\medskip

Let $f \in \mathcal{S}$ and let
\[
\log\frac{f(t)-f(z)}{t-z}=\sum_{p,q=0}^{\infty}\omega_{p,q}t^{p}z^{q},
\]
where $\omega_{p,q}$ are called Grunsky's coefficients with property $\omega_{p,q}=\omega_{q,p}$.
For those coefficients in \cite{duren,Lebedev} we can find the next Grunsky's inequality:
\be\label{eq 4}
\sum_{q=1}^{\infty}q \left|\sum_{p=1}^{\infty}\omega_{p,q}x_{p}\right|^{2}\leq
\sum_{p=1}^{\infty}\frac{|x_{p}|^{2}}{p},
\ee
where $x_{p}$ are arbitrary complex numbers such that last series converges.

\medskip

Further, it is well-known that if $f\in \mathcal{S}$ and has the form \eqref{eq 1}, then also
\[
f_{2}(z)=\sqrt{f(z^{2})}=z +c_{3}+c_{5}z^{5}+...
\]
belongs to the class $\mathcal{S}$. Then for the function $f_{2}$ we have the appropriate Grunsky's
coefficients of the form $\omega_{2p-1,2q-1}^{(2)}$ and the inequality \eqref{eq 4} gets the form
\be\label{eq 6}
\sum_{q=1}^{\infty}(2q-1) \left|\sum_{p=1}^{\infty}\omega_{2p-1,2q-1}^{(2)}x_{2p-1}\right|^{2}\leq
\sum_{p=1}^{\infty}\frac{|x_{2p-1}|^{2}}{2p-1}.
\ee
As it has been shown in \cite[p.57]{Lebedev}, if $f$ is given by \eqref{eq 1} then the coefficients $a_{2}$, $a_{3}$, $a_{4}$, $a_{5}$
are expressed by Grunsky's coefficients  $\omega_{2p-1,2q-1}^{(2)}$ of the function $f_{2}$ given by
\eqref{eq 3} in the following way (in the next text we omit upper index "(2)" in $\omega_{2p-1,2q-1}^{(2)}$):
\be\label{eq 7}
\begin{split}
a_{2}&=2\omega _{11},\\
a_{3}&=2\omega_{13}+3\omega_{11}^{2}, \\
a_{4}&=2\omega_{33}+8\omega_{11}\omega_{13}+\frac{10}{3}\omega_{11}^{3}\\
a_{5}&=2\omega_{35}+8\omega_{11}\omega_{33}+5\omega_{13}^{2}+18\omega_{11}^{2}\omega_{13}+\frac{7}{3}\omega_{11}^{4}\\
0&=3\omega_{15}-3\omega_{11}\omega_{13}+\omega_{11}^{3}-3\omega_{33}\\
0&=\omega_{17}-\omega_{35}-\omega_{11}\omega_{33}-\omega_{13}^{2}+\frac{1}{3}\omega_{11}^{4}.
\end{split}
\ee
We note that in \cite{Lebedev} there exists a typing mistake for the coefficient $a_{5}$. Namely,
instead of the therm $5\omega_{13}^{2}$, there is $5\omega_{15}^{2}.$

\medskip

Also, from \eqref{eq 6} for $x_{2p-1}=0$, $p=3,4,\ldots$, we have
\be\label{eq 8}
\begin{split}
&\quad |\omega_{11}x_{1}+\omega_{31}x_{3}|^{2}+3|\omega_{13}x_{1}+\omega_{33}x_{3}|^{2}
+|\omega_{15}x_{1}+\omega_{35}x_{3}|^{2}\\
&+|\omega_{17}x_{1}+\omega_{37}x_{3}|^{2}
\leq |x_{1}|^{2}+\frac{|x_{3}|^{2}}{3}.
\end{split}
\ee
From \eqref{eq 8}, having in mind that  $\omega_{31}=\omega_{13}$, for $x_{1}=1$ and $x_{3}=0 $ we have
the next inequalities
\[
\begin{split}
|\omega_{11}|^{2}&\leq1,\\
|\omega_{11}|^{2}+3|\omega_{13}|^{2}&\leq1 ,\\
|\omega_{11}|^{2}+3|\omega_{13}|^{2}+5|\omega_{15}|^{2}&\leq1 ,\\
|\omega_{11}|^{2}+3|\omega_{13}|^{2}+5|\omega_{15}|^{2}+7|\omega_{17}|^{2}&\leq1 .
\end{split}
\]
From the last inequalities we easily obtain
\be\label{eq 9}
\begin{split}
|\omega _{11}|& \leq1 ,\\
|\omega _{13}|&\leq\frac{1}{\sqrt{3}}\sqrt{1-|\omega _{11}|^{2}},\\
|\omega _{15}|&\leq\frac{1}{\sqrt{5}}\sqrt{1-|\omega _{11}|^{2}-3|\omega _{13}|^{2}},\\
|\omega _{17}|&\leq\frac{1}{\sqrt{7}}\sqrt{1-|\omega _{11}|^{2}-3|\omega _{13}|^{2}-5|\omega _{15}|^{2}}.
\end{split}
\ee
We note that we get the first inequality from \eqref{eq 9} also using the fact $|a_{2}|=|2\omega_{11}|\leq2$ (see \eqref{eq 7}).

\medskip

\section{Main results}

\medskip

We start with improvement of some results given in Theorem \ref{th-1}.

\begin{thm}\label{th-4}
Let $f\in\es$ and be given by \eqref{eq 1} with $a_2=0$. Then
\begin{itemize}
\item[($i$)]  $|a_5|\leq\frac{3}{4}+\frac{1}{\sqrt{7}}=1.12796\ldots $,
\item[($ii$)] $|H_3(1)|\leq 1.026\ldots$.
\end{itemize}
\end{thm}

\begin{proof}$\mbox{}$
\begin{itemize}
\item[($i$)] The classical inequality $|a_3-a_2^2|\le1$ for $f$ in $\es$ when $a_2=0$, gives  $|a_3|\le1$, which from \eqref{eq 7} gives
\be\label{eq 10}
|\omega_{13}|\le\frac12.
\ee
\medskip
Since $\omega_{11}=0$ ($\Leftrightarrow a_2=0$), from the relation for $a_5$ in \eqref{eq 7} and last relation in it, we obtain
\be\label{eq 11}
|a_5|=|2\omega_{35}+5\omega_{13}^2|
\ee
and
\be\label{eq 12}
\omega_{35}=\omega_{17}-\omega_{13}^2.
\ee
Further, using  the relations \eqref{eq 11} and \eqref{eq 12} we have
\be\label{eq 13}
\begin{split}
|a_5|&=|2\omega_{17}+3\omega_{13}^2|\\
&\leq 2|\omega_{17}|+3|\omega_{13}|^2\\
&\leq \frac{2}{\sqrt{7}}\sqrt{1-3|\omega _{13}|^{2}}+3|\omega_{13}|^2\\
&\leq \frac{3}{4}+\frac{1}{\sqrt{7}}=1.12796\ldots  ,
\end{split}
\ee
since by \eqref{eq 7}  ($\omega_{11}=0$),
\be\label{eq 14}
|\omega _{17}|\leq\frac{1}{\sqrt{7}}\sqrt{1-3|\omega _{13}|^{2}-5|\omega _{15}|^{2}}
\leq \frac{1}{\sqrt{7}}\sqrt{1-3|\omega _{13}|^{2}}
\ee
and $|\omega_{13}|\le\frac12$ by \eqref{eq 10}.

\medskip

\item[($ii$)]
When $\omega_{11}=0$,  the fifth relation in \eqref{eq 7} gives $\omega_{33}=\omega_{15},$ and using \eqref{eq 12}, from \eqref{eq 3} we have
\[
\begin{split}
|H_3(1)
&= |-8\omega_{13}^3-4\omega_{33}^2+2(2\omega_{35}+5\omega_{13}^{2})\omega_{13}|\\
&= |-8\omega_{13}^3-4\omega_{15}^2+4\omega_{13}(\omega_{17}-\omega_{13}^2)+10\omega_{13}^{3}|\\
&= |-2\omega_{13}^3-4\omega_{15}^2+4\omega_{13}\omega_{17}|\\
&\leq 2|\omega_{13}|^3+4|\omega_{15}|^2+4|\omega_{13}||\omega_{17}|\\
&\leq 2|\omega_{13}|^3+\frac{4}{5}(1-3|\omega_{13}|^{2})+\frac{4}{\sqrt{7}}|\omega_{13}|
\sqrt{1-3|\omega_{13}|^{2}}\\
&=: F_{1}(|\omega_{13}|),
\end{split}
\]
where
\[F_{1}(y)=2y^{3}+\frac{4}{5}(1-3y^{2})+\frac{4}{\sqrt{7}}y\sqrt{1-3y^{2}},\quad0\leq y\leq\frac{1}{2}.\]
Here we used the relations \eqref{eq 11} and \eqref{eq 14}. Now, using the first derivative test we conclude that the function $F_{1}$ attains its maximum for $y_{0}=0.286667\ldots$ with $F_{1}(y_{0})=1.026\ldots .$
\end{itemize}

\medskip

Finally, we note that the results (i) and (iv) in Theorem \ref{th-1} are the best possible as the function $f(z)=\frac{z}{1-z^{2}}$ shows.
\end{proof}

\medskip

We next prove a similar result, this time assuming that $a_3=0$.

\begin{thm}\label{th-5}
Let $f\in\es$ and be  given by \eqref{eq 1}, with $a_3=0$. Then
\medskip
\begin{itemize}
\item[($i$)] $|a_2|\le1$,
\medskip
\item[($ii$)] $|a_4|\le \frac14\sqrt{\frac{21}{5}}+\frac{5}{8} = 1.1373\ldots $,
\medskip
\item[($iii$)]  $|a_5|\le 1.674896577\ldots$,
\medskip
\item[($iv$)] $|H_2(2)|\le 1.1373\ldots$,
\medskip
\item[($v$)] $|H_3(1)|\le 0.6647958756\ldots$.
\end{itemize}
\end{thm}
\medskip

\begin{proof}$\mbox{}$
\begin{itemize}
\item[($i$)] We imitate the proof of Theorem \ref{th-4}(i) and from  $|a_3-a_2^2|\le1$,  $a_3=0$, we receive $|a_2^2|\le1$, i.e., $|a_2|\le1$.

\medskip

Further we will use that from \eqref{eq 7} follows $a_3=2\omega_{13}+3\omega_{11}^2=0$, and then
\be\label{eq 16}
\omega_{13}=-\frac32 \omega_{11}^2 \quad \left(\Leftrightarrow\,\, \omega_{11}^2=-\frac23\omega_{13}\right),
\ee
and also, from $|a_2|=|2\omega_{11}|\le1$,
\[
|\omega_{11}|\le \frac12 \quad\mbox{and}\quad |\omega_{13}|\le \frac38 \,\,(\mbox{by } \eqref{eq 16}.
\]
%
\medskip

\item[($ii$)] By using \eqref{eq 7} and \eqref{eq 16}, we obtain
\be\label{eq 18}
|a_4|
= \left|2\omega_{33}+8\omega_{11} \left(-\frac32\omega_{11}^2\right)+\frac{10}{3}\omega_{11}^3\right|= \left|2\omega_{33}-\frac{26}{3}\omega_{11}^3\right|.
\ee
On the other hand, using the fifth relation in \eqref{eq 7}, and \eqref{eq 16}, we have
\be\label{eq 19}
\begin{split}
\omega_{33}&=\omega_{15}-\omega_{11}\omega_{13}+\frac{1}{3}\omega_{11}^{3}=\omega_{15}-\omega_{11}(-\frac32\omega^2_{11})+\frac{1}{3}\omega_{11}^{3}\\
&=\omega_{15}+\frac{11}{6}\omega_{11}^{3}.
\end{split}
\ee
Combining \eqref{eq 18} and \eqref{eq 19} we obtain
\be\label{eq 20}
\begin{split}
|a_4|
&=|2\omega_{15} -5 \omega_{11}^{3}|\leq 2|\omega_{15}|+5|\omega_{11}|^{3}\\
&\leq \frac{2}{\sqrt{5}}\sqrt{1-|\omega_{11}|^{2}-\frac{27}{4}|\omega_{11}|^{4}}+5|\omega_{11}|^{3}\\
&=:F_{2}(|\omega_{11}|),
\end{split}
\ee
where
\[F_{2}(x)=\frac{2}{\sqrt{5}}\sqrt{1-x^{2}-\frac{27}{4}x^{4}}+5x^{3},\quad  0\leq x \leq \frac{1}{2},\]
and where we used \eqref{eq 9} and \eqref{eq 16} to obtain
\be\label{eq 21}
|\omega_{15}|\leq \frac{1}{\sqrt{5}}\sqrt{1-|\omega_{11}|^{2}-\frac{27}{4}|\omega_{11}|^{4}}.
\ee
Finally, $F_{2}(x)$ is strictly increasing function on the interval $[0,1/2]$, attaining its maximal value $\frac14\sqrt{\frac{21}{5}}+\frac{5}{8} = 1.1373\ldots$ for  $x=1/2$.

\medskip
\item[($iii$)]
From the last relation in \eqref{eq 7}, using \eqref{eq 16} and \eqref{eq 19}, after simple calculation we receive
\be\label{eq 22}
\omega_{35}=\omega_{17}-\omega_{11}\omega_{15}-\frac{15}{4}\omega_{11}^{4}.
\ee
Using the relations \eqref{eq 7} and \eqref{eq 22}, we get
\be\label{eq 23}
\begin{split}
|a_5|
&= \left|2\omega_{17} + 6\omega_{11}\omega_{15}  - \frac{25}{4}\omega_{11}^4\right|\\
&\leq 2|\omega_{17}| + 6|\omega_{11}||\omega_{15}|  +\frac{25}{4}|\omega_{11}|^4 \\
&\leq \left( \frac{2}{\sqrt{7}}+\frac{6}{\sqrt{5}}|\omega_{11}|\right)\sqrt{1-|\omega_{11}|^{2}-\frac{27}{4}|\omega_{11}|^{4}} +\frac{25}{4}|\omega_{11}|^{4}\\
&=:F_{3}(x),
\end{split}
\ee
where
\[F_{3}(x)=\left(\frac{2}{\sqrt{7}}+\frac{6}{\sqrt{5}}x \right) \sqrt{1-x^{2}-\frac{27}{4}x^{4}}+\frac{25}{4}x^{4},\quad  0\leq x \leq \frac{1}{2}, \]
and where we used the estimate  given in \eqref{eq 21}, the estimate from \eqref{eq 9}, and
\be\label{eq 24}
\begin{split}
|\omega _{17}| &\leq\frac{1}{\sqrt{7}}\sqrt{1-|\omega _{11}|^{2}-3|\omega _{13}|^{2}-5|\omega _{15}|^{2}}\\
&\leq \frac{1}{\sqrt{7}}\sqrt{1-|\omega _{11}|^{2}-3|\omega _{13}|^{2}}.
\end{split}
\ee
Now, the first derivative test shows that on the interval $[0,1/2]$, the function $F_3$ has maximal value 1.674896577\ldots attained for $x=0.43957885\ldots$.
%
\medskip
\item[($iv$)] Since $a_{3}=0$ using \eqref{eq 2} we have $H_2(2)=a_{2}a_{4}$
and then from the relation \eqref{eq 18} and estimation \eqref{eq 21}, we obtain
\[
\begin{split}
|H_2(2)|&=|a_{2}a_{4}|=|2\omega_{11}(2\omega_{15}-5\omega_{11}^{3})\\
&\leq 4|\omega_{11}||\omega_{15}|+ 10|\omega_{11}|^{4}\\
&\leq \frac{4}{\sqrt{5}}|\omega_{11}|\sqrt{1-|\omega_{11}|^{2}-\frac{27}{4}|\omega_{11}|^{4}}+10|\omega_{11}|^{4}\\
&=:F_{4}(|\omega_{11}|),
\end{split}
\]
where
\[F_{4}(x)=\frac{4}{\sqrt{5}}x\sqrt{1-x^{2}-\frac{27}{4}x^{4}}+10x^{4},\quad 0\leq x \leq \frac{1}{2}.\]
Again, the first derivative test shows that  $F_4$ is strictly increasing on the interval $[0,1/2]$ with maximal value 1.1373\ldots attained for $x=1/2$.

\medskip

\item[($v$)] After applying  $a_{3}=0$ in the definition \eqref{eq 3} we receive
\be\label{eq 26}
H_3(1)=-a_{4}^{2}-a_{5}a_{2}^{2}.
\ee
Recall that in \eqref{eq 20} and \eqref{eq 23} we found $a_{4}=2\omega_{15}-5\omega_{11}^{3}$ and $a_{5}=2\omega_{17} + 6\omega_{11}\omega_{15}  - \frac{25}{4}\omega_{11}^4$, respectively.
From these facts and \eqref{eq 26}, after some calculations we obtain
\be\label{eq 27}
H_3(1) = - \omega_{15}^2 -4 \omega_{11}^3 \omega_{15} - 8 \omega_{11}^2 \omega_{17}.
\ee
So, using \eqref{eq 27} and estimates given in \eqref{eq 21} and \eqref{eq 24} we get
\[
\begin{split}
|H_3(1)| &  =  |\omega_{15}|^2 + 4 |\omega_{11}|^3 |\omega_{15}| +  8 |\omega_{11}|^2 |\omega_{17}| \\
& \leq \frac{1}{5} \left(1-|\omega_{11}|^{2}-\frac{27}{4}|\omega_{11}|^{4} \right) \\
&\quad + \left(\frac{4}{\sqrt{5}} |\omega_{11}|^{3} + \frac{8}{\sqrt{7}}|\omega_{11}|^{2}\right) \sqrt{1-|\omega_{11}|^{2}-\frac{27}{4}|\omega_{11}|^{4}}\\
&=:F_{5}(|\omega_{11}|),
\end{split}
\]
where
\[
\begin{split}
F_{5}(x) &= \frac{1}{5}\left(1-x^{2}-\frac{27}{4}x^{4}\right)+ \left(\frac{4}{\sqrt{5}}x^{3}+\frac{8}{\sqrt{7}}x^{2}\right)\sqrt{1-x^{2}-\frac{27}{4}x^{4}},
\end{split}
\]
$0\leq x \leq \frac{1}{2}$. The first derivative test shows that the function $F_{5}(x)$ when $0\leq x \leq 1/2$, has maximal value 0.6647958756\ldots attained for $x=0.458573\ldots$.
\end{itemize}
\end{proof}

\medskip

*************************** DTS

\begin{thm}\label{th-6}
Let $f\in\es$ and be given by \eqref{eq 1}.
\begin{itemize}
  \item[(i)] Then $|a_4|-|a_3| \le 1.75185\ldots$.
  \item[(ii)] If $f$ is an odd function, then $|a_5|-|a_3| \le \frac{2}{\sqrt7}=0.7559\ldots.$
\end{itemize}
\end{thm}

\begin{proof}$ $
\begin{itemize}
\item[(i)] Using \eqref{eq 7} and  $|\omega_{11}|\leq 1$, we have
\be\label{eq 8-2}
|a_4|-|a_3| \le |a_4|-|\omega_{11}||a_3| \le  |a_4-\omega_{11} a_3| = 2\Big|\omega_{33}
+3\omega_{11}\omega_{13} + \frac16 \omega_{11}^3\Big|.
\ee
From the fifth relation in \eqref{eq 7} we obtain
\be\label{eq 9-2}
\omega_{33}=\omega_{15}-\omega_{11}\omega_{13}+\frac{1}{3}\omega_{11}^{3},
\ee
and using the relations  \eqref{eq 8-2} and \eqref{eq 9-2}, after some simple calculations we have
\[
\begin{split}
|a_4|-|a_3|& \leq|2\omega_{15}+4\omega_{11}\omega_{13}+\omega_{11}^{3}|\\
&\leq |\omega_{15}|+4|\omega_{11}||\omega_{13}|+|\omega_{11}|^{3}\\
&\leq \frac{2}{\sqrt{5}}\sqrt{1-|\omega_{11}|^{2}-3|\omega_{13}|^{2}}+4|\omega_{11}||\omega_{13}|
+|\omega_{13}|^{3}\\
&=: F_{6}(|\omega_{11}|,|\omega_{13}|,
\end{split}
\]
where
\[
F_{6}(x,y)=\frac{2}{\sqrt{5}}\sqrt{1-x^{2}-3y^{2}}+4xy+x^{3},
\]
and $0\leq x \leq1$, $0\leq y \leq\frac{1}{\sqrt{3}}\sqrt{1-x^{2}}.$

\medskip

Now, we need to find maximum of the function  $F_{6}(x,y)$ when
\[(x,y)\in \left\{(x,y):0\leq x \leq1, 0\leq y \leq\frac{1}{\sqrt{3}}\sqrt{1-x^{2}}\right\}:=D_{1}.\]
We start the analysis in the interior of $D_1$. From
\[ \frac{\partial F_6(x,y)}{\partial x} = -\frac{2 x}{\sqrt{5} \sqrt{1-x^2-3 y^2}}+3 x^2+4 y \]
and
\[ \frac{\partial F_6(x,y)}{\partial y} = 4 x-\frac{6 y}{\sqrt{5} \sqrt{1-x^2-3 y^2}},\]
we receive that the stationary points (it existing) satisfy
\[ 3y\frac{\partial F_6(x,y)}{\partial x} - x\frac{\partial F_6(x,y)}{\partial y} = x^2 (9 y-4)+12 y^2 = 0, \]
i.e., on $D_1$,
\[ x = \frac{2 \sqrt{3} y}{\sqrt{4-9 y}}.\]
Finally, substituting $\frac{2 \sqrt{3} y}{\sqrt{4-9 y}}$ for $x$ in $\frac{\partial F_6(x,y)}{\partial x} = 0$, and numerically solving the corresponding equation for $y$, we receive solution $y_{01} = 0.2872\ldots$, and further $x_{01} = \frac{2 \sqrt{3} y_{01}}{\sqrt{4-9 y_{01}}} = 0.83634\ldots$. It is easy to check that $(x_{01},y_{01})\in D_1$ and that $F_6(x_{01},y_{01}) = 1.75185\ldots.$

On the edges of $D_1$ we have:
\begin{itemize}
  \item[-] $F_6(0,y) = \frac{2}{\sqrt{5}} \sqrt{1-3 y^2}$, with maximal value for $y\ge0$, $\frac{2}{\sqrt{5}} = 0.8944\ldots$;
  \item[-] $F_6(x,0) = \frac{2}{\sqrt{5}} \sqrt{1-x^2} + x^3$, with maximal value for $0\le x\le1$ equaling to $1.13666\ldots$ for $x=0.9494\ldots$;
  \item[-] $F_6\left(x, \frac{\sqrt{1-x^2}}{\sqrt{3}}\right) = \frac{4}{\sqrt{3}} x\sqrt{1-x^2} + x^3$, with maximal value for $0\le x\le1$ equaling to $1.6496\ldots$ for $x=0.8628\ldots$.
\end{itemize}
So, the function $F_6$ attains its maximal value $1.75185\ldots$ in the interior point $(x_{01},y_{01})$ of its domain $D_1$ and the conclusion of part (i) follows.

\medskip

\item[(ii)] Since $f$ is odd, $a_2=0$, and then using $|a_3-a_2^2|\le 1$, we have $|a_3|=|2\omega_{13}| \le 1$, i.e., $|\omega_{13}| \le \frac12$. Further, from \eqref{eq 7}, \eqref{eq 12} and \eqref{eq 13}, we receive
\[ |a_5| = |2\omega_{17} +3\omega_{13}^2|. \]
Also, since $a_4=0$, from \eqref{eq 7} follows $\omega_{33}=0$ and $\omega_{15}=0$. So, since $|a_3|=|2\omega_{13}| \le 1$,
\[
\begin{split}
|a_5|-|a_3| &\le |a_5|-2|\omega_{13}||a_3| =  |2\omega_{17} + 3\omega_{13}^2| - 4|\omega_{13}|^2\\
&\le |(2\omega_{17} + 3\omega_{13}^2)-4\omega_{13}^2| = |2\omega_{17} - \omega_{13}^2|\\
& \le 2|\omega_{17}| + |\omega_{13}|^2 \le \frac{2}{\sqrt7}\sqrt{1-3|\omega_{13}|^2} + |\omega_{13}|^2\\
&\le \frac{2}{\sqrt7},
\end{split}
\]
where $|\omega_{13}|\le \frac12.$
\end{itemize}
\end{proof}

\begin{rem}
The estimate given in Theorem \ref{th-6}(ii) is an improvement of the result $|a_5|-|a_3| < 1$ given in \cite[p. 17]{lecko}.
\end{rem}

\end{document}